\documentclass[reqno]{amsart}
\usepackage{amssymb}
\usepackage{amsmath}
\usepackage{amsthm}
\usepackage{mathtools}
\usepackage{graphicx}
\usepackage{comment}
\newtheorem{thm}{Theorem}
\newtheorem*{thmun}{Theorem 3}
\newtheorem{prop}[thm]{Proposition}
\newtheorem{lem}[thm]{Lemma}

\theoremstyle{definition}

\theoremstyle{remark}

\newcommand{\BR}{\mathbb{R}}
\newcommand{\BZ}{\mathbb{Z}}

\begin{document}

\title{The Inhomogeneous Wave Equation with $L^p$ Data}

\author{Ben Foster}
\address{Department of Mathematics, University of Pennsylvania, 
Philadelphia, PA 19104}
\email{bfost@sas.upenn.edu}

\begin{abstract}
We prove existence and uniqueness of $L^2$ solutions to the inhomogeneous wave equation on $\BR^{n-1}\times\BR$ under the assumption that the inhomogeneous data lies in $L^p(\BR^n)$ for $p=2n/(n+4)$. We also require the Fourier transform of the inhomogeneous data to vanish on an infinite cone where the solution could become singular. Subsequently, we show sharpness of the exponent $p$. This extends work of Michael Goldberg, in which similar Fourier-analytic techniques were used to study the inhomogeneous Helmholtz equation.
\end{abstract}

 \maketitle
\section{Introduction}
We consider solutions to the inhomogeneous wave equation via Fourier-analytic methods
\begin{equation}
\left\{
\begin{array}{c l}	
     u_{tt}-\Delta_x u = f, & f\in L^p(\BR^{n})\\
     u\in L^2(\BR^n)
\end{array}\right.
\end{equation}
where we view $\BR^n=\BR^{n-1}\times \BR$ as having $n-1$ spatial dimensions. Taking the space-time Fourier transform, which we will denote throughout using the notation $\hat{u}$, we obtain a dual formulation
\begin{equation}
\left\{
\begin{array}{c l}	
     \hat{u}(\xi,\tau)= \dfrac{1}{4\pi^2}\dfrac{\hat{f}(\xi,\tau)}{|\xi|^2-\tau^2} \vspace{0.1cm}\\ 
     \hat{u}\in L^2(\BR^n)
\end{array}\right.
\end{equation}
Here and in the sequel, we use the following definition for the Fourier transform
\begin{equation}
\hat{u}(\xi,\tau)=\int_{\BR^n}u(x,t)e^{-2\pi i(x\cdot\xi+t\tau)}dxdt.
\end{equation}
From examining the formula for $\hat{u}$, a plausible strategy is to bound the $L^2$ norm of $u$ in terms of the $L^p$ norm of $f$. We notice immediately that if our solution is in $L^2$, then it is unique as a result of this formulation and Fourier inversion. The solution is at risk of blowing up along the infinite cone where $|\xi|=|\tau|$, so we will require that $\hat{f}$ vanishes in a suitable sense on this set. It is helpful to note that when studying the homogeneous version of this problem, taking the Fourier transform reveals that the solution $\hat{u}$ is supported on the infinite cone, in contrast to the situation in this problem. \\
A similar problem was studied by Michael Goldberg in \cite{MR3621102} with the PDE in question being the inhomogeneous Helmholtz equation. For that problem, the submanifold on which the Fourier transform of the solution needs to vanish is the unit sphere, which enjoys the nice properties of being convex, compact, and having nonvanishing curvature at every point. We utilize many of the techniques from his paper, although certain features of the wave equation require different techniques. Fortunately, the wave equation enjoys a scaling homogeneity in space and time; as a result, the problem is well suited to being studied on a single annulus in frequency space using a Littlewood-Paley decomposition and rescaling. \\
We prove two preliminary results before our main result. First is a lemma estimating the decay of the Fourier transform of surface measures for level sets of the function $g(\xi,\tau)=|\xi|^2-\tau^2$. These bounds are key in developing an $L^2$ bound for the Fourier transform of the solution localized to an annulus. After obtaining a suitable technical estimate on a single annulus in Proposition 2, we quickly conclude our main result:
\begin{thmun}
Suppose $f\in L^p(\BR^n)$ where $n\geq 5$ and $1< p\leq \frac{2n}{n+4}$ and that its Fourier transform $\hat{f}$ vanishes on the cone $\{(\xi,\tau)\in\BR^{n-1}\times\BR:|\xi|^2-\tau^2=0\}$. Then the inhomogeneous wave equation
\begin{equation} \label{Inhom Wave}
u_{tt}-\Delta_xu=f,
\end{equation}
where $\Delta_x$ denotes the Laplacian in the $n-1$ spatial variables, has a solution $u$ which is a tempered distribution whose Fourier transform coincides with a function, which is the unique solution with the property
\begin{equation}
    (-\Delta)^{\frac{1}{4}\left(n+4-\frac{2n}{p}\right)}u\in L^2(\BR^n),
\end{equation}
where $(-\Delta)^z$ denotes the Riesz potential. The Fourier transform of the Riesz potential of the solution is given by
\begin{equation} \label{Solution}
\left[(-\Delta)^{\frac{1}{4}\left(n+4-\frac{2n}{p}\right)}u\right]^{\land}(\xi,\tau)=(|\xi|^2+\tau^2)^{\frac{1}{4}\left(n+4-\frac{2n}{p}\right)}\frac{\hat{f}(\xi,\tau)}{|\xi|^2-\tau^2}.
\end{equation}
In particular, if $p=\frac{2n}{n+4}$ then $u\in L^2$.
\end{thmun}
In this case, the definition we are using for the Riesz potential $(-\Delta)^{-z/2}$ where $z>0$ is given by
\begin{equation}
(-\Delta)^{-z/2}(u)(x,t)=\frac{1}{c_z}\int_{\BR^n}\frac{u(y,s)}{|(x,t)-(y,s)|^{n-z}}dyds,
\end{equation}
where $c_z$ is a constant depending on $z$ and $n$. In the theorem, the hypothesis that $1<p\leq\frac{2n}{n+4}$ will ensure that the convolution kernel of the Riesz potential is a locally integrable function. The Riesz potentials are used because their Fourier multipliers are powers of the function measuring distance from the origin; this makes them suitable in cancelling out factors that would keep the sum of the Littlewood-Paley projections from converging. \\
Finally, in Section 3 of the paper, we give an example to show that the estimate on an annulus in Proposition 2 is sharp in $p$. This comes down to a standard dilation argument. The same example can be used to show that the solution $u$ is generally not in $L^2$ unless $p=\frac{2n}{n+4}$. \\
Goldberg's work has recently found applications in \cite{MR3983048} and \cite{MR3949567}, which studied various Schr\"odinger operators given by the Laplacian with an additional potential. This suggests that there may be applications of our main result to wave operators with potentials. \\
An interesting feature of the argument we use is that we associate a partial differential operator with a submanifold of Euclidean space, in this case the cone, and in Goldberg's case, the sphere. More generally, we can associate many linear partial differential operators with constant coefficients to subsets of Euclidean space via a Fourier duality, i.e.
\begin{equation}
\sum_{|\alpha|\leq N}c_{\alpha}\partial_{\alpha}u=f \qquad \longleftrightarrow \qquad\left\{x\in \BR^n: \sum_{|\alpha|\leq N}c_{\alpha}(2\pi i x)^{\alpha}=0\right\}.
\end{equation}
It is not unreasonable to suspect that when these associated subsets are submanifolds, then data about their curvature, convexity, and compactness could be used to obtain analogous results as the ones in this paper for other partial differential equations.
\section*{Acknowledgements}
The author would like to thank Philip Gressman for introducing him to the problem and for many helpful discussions.
\section{Main Result}
In this section, we will prove the main theorem in three steps. First, we develop some notation that will be used throughout the proof. Since our primary goal is to establish boundedness of operators, we will use the notation $A\lesssim B$ to mean there exists some constant $c$ such that $A\leq cB$; this constant may depend on certain parameters (such as the dimension $n$) which will be specified. We will use the notation $A\approx B$ to mean that $A\lesssim B$ and $B\lesssim A$. \\
As we will be performing a Littlewood-Paley decomposition, let $\chi$ be a radial Schwartz function which is supported on the set $A=\{z\in\BR^n:1/3\leq |z|\leq 1\}$ and satisfies
\begin{equation}
    \sum_{j\in \BZ}|\chi(2^{-j}\xi,2^{-j}\tau)|^2=1,\qquad\text{for }(\xi,\tau)\in\BR^n\setminus\{0\}.
\end{equation}
Denoting $\chi_j(\xi,\tau)=\chi(2^{-j}\xi,2^{-j}\tau)$, let $P_j$ be the operator whose Fourier multiplier is $\chi_j$, i.e. $\widehat{P_jf}=\chi_j\hat{f}$. \\
The main technical estimate will be over the level sets of the function
\begin{equation}
    g(x,t)=|x|^2-t^2.
\end{equation}
Let $\sigma_s$ denote the canonical surface measure on the embedded hypersurface $g^{-1}(s)$; when $s=0$, we must technically delete the origin in order to have an embedded submanifold, but this will not affect the proof at all. We will need to smoothly localize these measures to slightly larger annuli, so we define $\eta$ to be a smooth radial bump function that is identically 1 on the support of $\chi$ and is supported on the annulus of points at distances $1/4$ to $5/4$ from the origin. We let $\eta_j(\xi,\tau)=\eta(2^{-j}\xi,2^{-j}\tau)$. We will denote our smoothly truncated measures by
\begin{equation}
\sigma_s^{(j)}=\eta_j\sigma_s.
\end{equation}
First, we prove a lemma showing that the Fourier transform of a localized conical measure decays in time sufficiently quickly. The general strategy here is to write the Fourier transform of the measure as an oscillatory integral whose phase function has a nondegenerate Hessian. Since the cone has a single vanishing principal curvature at every point, however, we must first use the coarea formula to reduce to a manifold one dimension lower on which the phase function has a nondegenerate Hessian.
\begin{lem}
The following bound holds uniformly across all $x\in\BR^n$ and across $s\in(-1/10,1/10)$:
\begin{equation}
|\check{\sigma}_s^{(0)}(x,t)|\lesssim(1+|t|)^{\frac{2-n}{2}}
\end{equation}
\end{lem}
\begin{proof}
Since the function is radial, we can without loss of generality write $x=\lambda e_1$, where $e_1=(1,0,...,0)$. By linearity, it suffices to work with the positive time branch of the surface. Here, we have the graph parametrization for $g^{-1}(s)$ given by
\begin{equation}
\Phi_s(\xi_1,...,\xi_{n-1})=(\xi_1,...,\xi_{n-1},\sqrt{|\xi|^2-s}).
\end{equation}
When pulling back the surface integral to Euclidean space, we will get a uniformly bounded factor from the derivative of the graph parametrization that we can absorb as a constant. We let $h_s(\xi)=\sqrt{|\xi|^2-s}$ and let $k_s(\xi)$ for the derivative factor of $h_s$ accrued when pulling back the surface integral to Euclidean space. Using the notation $\xi'=(\xi_2,...,\xi_{n-1})$, we compute:
\begin{align}
\check{\sigma}_s^{(0)}(x,t)&=\int_{\BR^{n-1}}e^{2\pi i(\lambda \xi_1+th_s(\xi))}k_s(\xi)\eta(\xi,h_s(\xi))d\xi \\
&=\int_{-1}^1e^{2\pi i\lambda a}\left(\int_{\{\xi_1=a\}}e^{2\pi ith_s(a,\xi')}k_s(a,\xi')\eta(a,\xi',h_s(a,\xi'))d\xi'\right)da
\end{align}
The inner integral is an oscillatory integral with phase function $h_s$. This has derivative (with respect to the $\xi'$ variables) $\xi'(a^2+|\xi'|^2-s)^{-1/2}$ which is nonvanishing except when $\xi'=0$. We check that the Hessian is nondegenerate at such points:
\begin{align}
\partial_i^2h_s(a,\xi')&=\frac{\sqrt{a^2+|\xi'|^2-s}-\xi_i^2/\sqrt{a^2+|\xi'|^2-s}}{a^2+|\xi'|^2-s} \\
\partial_i\partial_jh_s(a,\xi')&=-\frac{\xi_i\xi_j}{(a^2+|\xi'|^2-s)^{3/2}}
\end{align}
Thus, evaluating at $\xi'=0$ we see that the Hessian is nondegenerate provided $a^2\neq s$; however, if $a^2=s$ and $\xi'=0$ then we are outside the support of $\eta$, so there is no issue. Using stationary phase techniques described in Theorem 1 of Chapter 8.3 of \cite{MR1232192}, we conclude that
\begin{equation}
\left|\int_{\{\xi_1=a\}}e^{2\pi it\sqrt{a^2+|\xi'|^2-s}}\eta(a,\xi',\sqrt{a^2+|\xi'|^2-s})d\xi'\right|\lesssim(1+|t|)^{(2-n)/2}.
\end{equation}
Thus, by the triangle inequality, we conclude the claim.
\end{proof}
With this estimate established, we are almost ready to prove the main technical estimate. It is not hard to see that it is necessary that $\hat{f}=0$ on the set where $|\xi|=|\tau|$; if $\hat{f}$ is a continuous function that does not have this property, then the integral
\begin{equation}
    \int\left|\frac{1}{|\xi|^2-\tau^2}\right|d\xi\,d\tau
\end{equation}
can be seen to diverge when the integral is taken over a neighborhood of the point on the set $|\xi|=|\tau|$ with the property that $\hat{f}(\xi,\tau)\neq 0$. One easy choice for such a neighborhood would be a thickened sector of a spherical shell; then we would have
\begin{equation}
    \int_S\left|\frac{1}{|\xi|^2-\tau^2}\right|d\xi\,d\tau=C_n\int_{r_0-\delta}^{r_0+\delta}\int_{r-\epsilon}^{r+\epsilon}\left|\frac{1}{r^2-t^2}\right|dtdr\rightarrow \infty
\end{equation}
so a vanishing condition is necessary. \\
The last remaining task is to clarify what it means for $\hat{f}$ to vanish on the cone $g^{-1}(0)$ for arbitrary $f\in L^p(\BR^n)$. The natural way to do this is by density. We say that $\hat{f}$ vanishes on the cone if there exists a sequence of functions $\{f_k\}\subset L^1(\BR^n)\cap L^p(\BR^n)$ which converges to $f$ with respect to the $L^p$ norm and satisfies that $\hat{f_k}=0$ on the set $\{(\xi,\tau)\in \BR^n:|\xi|^2=\tau^2\}$. This is unambiguous since each $\hat{f_k}$ is continuous. For simplicity, we will assume in the proof of the following proposition that $f$ is actually an element of the subspace of functions in $L^1\cap L^p$ whose continuous Fourier transform vanishes on the cone. In fact, when we prove boundedness on this subspace, we will obtain a unique bounded extension to the completion of this subspace with respect to the $L^p$ norm; we will consider elements in this completion to vanish on the cone as well.
\begin{prop}
Suppose $f\in L^p(\BR^n)$ where $1< p\leq \frac{2n}{n+4}$ and that its Fourier transform $\hat{f}$ vanishes on the cone $g^{-1}(0)$. Then we have the estimate
\begin{equation} \label{Complete estimate on annulus}
    \Vert P_0u\Vert_{L^2(\BR^n)}\lesssim \Vert P_0 f\Vert_{L^p(\BR^n)}
\end{equation}
where $P_0u$ has its Fourier transform given by
\begin{equation}
\widehat{P_0u}(\xi,\tau)=\frac{\hat{f}(\xi,\tau)}{|\xi|^2-\tau^2}\chi(\xi,\tau).
\end{equation}
\end{prop}
\begin{proof}
By Plancherel, it suffices to estimate $\Vert \widehat{P_0 u}\Vert_{L^2}$. We will split this integral over the level sets of $g$ which are close to the cone and far from the cone. Define $L=g^{-1}(-\delta,\delta)\cap A$ and fix $\delta\leq1/10$. The motivation for this is that for $|x|\geq 1/3$ we have that $|x|^2-s\geq 1/90>0$ which will make subsequent functions related to the graph parametrizations of $g^{-1}(s)$ smooth up to the boundary of the level sets on the annulus. Away from the cone, we can use the coarea formula to estimate
\begin{align}
\Vert \widehat{P_0u}\Vert_{L^2(A\setminus L)}^2&=\int_{A\setminus L}\frac{|\widehat{P_0f}(\xi,\tau)|^2}{(|\xi|^2-\tau^2)^2}d\xi d\tau\\
&\lesssim \int_{s=-\delta}^{\delta}\int_{g^{-1}(s)}\frac{|\widehat{P_0f}(\xi,\tau)|^2}{s^2}d\sigma_s(\xi,\tau)ds \\
&\lesssim \Vert \widehat{P_0 f}\Vert_{L^2(A\setminus L)}^2\\
&\lesssim \vert \widehat{P_0 f}\Vert_{L^{p'}(A\setminus L)}^2
\end{align}
where we used the fact that the denominator gives a factor of at most $\delta^{-2}$. As a consequence, we get the desired estimate on $A\setminus L$
\begin{equation}\label{Bound away from cone}
\Vert \widehat{P_0u}\Vert_{L^2(A\setminus L)}^2\lesssim \vert P_0 f\Vert_{L^{p}(A\setminus L)}^2.
\end{equation}
The more difficult problem is to estimate the integral over the level sets close to the cone. We will view the problem within the framework of convolution with a kernel. Throughout, it will suffice to work only on the positive-time branch of the cone and thus on the positive time half of the annulus; this is because the same methods can be used to derive an analogous estimate for the negative-time branch, contributing no more than a factor of 2. Define the family of kernels $K_0^{\epsilon}(x,t)$ (suppressing $(x,t)$ for succinctness) by
\begin{equation}\label{Kernel definition}
K_0^{\epsilon}=\int_{-\delta}^{\delta}\frac{(\sigma_s^{(0)})^{\lor}-(\sigma_0^{(0)})^{\lor}}{s^2+\epsilon^2}ds=\frac{1}{2}\int_{-\delta}^{\delta}\frac{(\sigma_s^{(0)})^{\lor}-2(\sigma_0^{(0)})^{\lor}+(\sigma_{-s}^{(0)})^{\lor}}{s^2+\epsilon^2}ds
\end{equation}
where the last equality follows by changing variables and averaging. This is defined for $\epsilon>0$ so that $K_0^{\epsilon}$ is always a tempered distribution. For estimation purposes, we will use the last definition of $K_0^{\epsilon}$ given in (\ref{Kernel definition}) to exploit symmetry. The motivation for this definition is the observation that 
\begin{equation}\label{Convolution form}
\langle K_0^{\epsilon}*P_0f,P_0f\rangle=\langle \widehat{K_0^{\epsilon}}\widehat{P_0f},\widehat{P_0f}\rangle=\int_{s=-\delta}^{\delta}\int_{g^{-1}(s)}\frac{|\widehat{P_0f}(\xi,\tau)|^2}{s^2+\epsilon^2}d\sigma_s^{(0)}(\xi,\tau)ds
\end{equation}
where we used the fact that $\hat{f}$ vanishes on the cone, resulting in no contribution from the $d\sigma_0$ measure. We were able to replace $d\sigma_s$ by $d\sigma_s^{(0)}$ because these measures are the same on the support of $\widehat{P_0f}$. By the coarea formula, we get the estimate
\begin{equation}\label{Convolution form estimate}
\langle K_0^{\epsilon}*P_0f,P_0f\rangle\approx\int_{L}\frac{|\widehat{P_0f}(\xi,\tau)|^2}{s^2+\epsilon^2}d\xi\,d\tau
\end{equation}
so we can estimate the integral of $u$ over $L$ by getting a $(p,p')$ estimate for the operator given by convolution with $K_0^{\epsilon}$. Now, we will show that we have a decay estimate for $K_0^{\epsilon}$ in the time variable which is uniform in $x$. From Lemma 1, we have the estimate
\begin{equation}\label{FT Measure estimate}
|\check{\sigma}_s^{(0)}(x,t)|\lesssim (1+|t|)^{\frac{2-n}{2}}.
\end{equation}
We can also differentiate under the integral sign in $s$ since the bump function $\eta$ cuts out all the singularities; this gives the estimate
\begin{equation}\label{FT Derivative Measure estimate}
|\partial_s^2\check{\sigma}_s^{(0)}(x,t)|\lesssim (1+|t|)^{\frac{2-n}{2}+2}.
\end{equation}
Applying the mean value theorem twice, we thus obtain the bound
\begin{equation}\label{Kernel numerator MVT estimate}
|(\check{\sigma}_s^{(0)}-2\check{\sigma}_0^{(0)}+\check{\sigma}_{-s}^{(0)})(x,t)|\lesssim (1+|t|)^{\frac{2-n}{2}+2}s^2.
\end{equation}
Combining (\ref{FT Measure estimate}) and (\ref{Kernel numerator MVT estimate}) we get that
\begin{equation}
|(\check{\sigma}_s^{(0)}-2\check{\sigma}_0^{(0)}+\check{\sigma}_{-s}^{(0)})(x,t)|\lesssim (1+|t|)^{\frac{2-n}{2}}\min\big((1+|t|)^2s^2,1\big).
\end{equation}
Splitting the integral around $|s|=(1+|t|)^{-1}:=b_t$, we conclude
\begin{align}
|K_0^{\epsilon}(x,t)|&\lesssim(1+|t|)^{\frac{2-n}{2}+2}\int\limits_{|s|<b_t}ds+(1+|t|)^{\frac{2-n}{2}}\int\limits_{\delta>|s|>b_t}\frac{ds}{s^2} \\
&\lesssim (1+|t|)^{\frac{2-n}{2}+1}\label{1,inf partial estimate}
\end{align}
with a constant that is uniform in $\epsilon$ and in $x$. \\
Define a family of partial-convolution operators $T_t^{(\epsilon)}$ for $g\in L^1(\BR^{n-1})+L^2(\BR^{n-1})$ by
\begin{equation}
T_t^{(\epsilon)}(g)(x)=\int_{\BR^{n-1}}K_0^{\epsilon}(x-y,t)g(y)dy.
\end{equation}
The motivation of this definition is that if we can first get a $(p,p')$ estimate for each $T_t^{(\epsilon)}$ operator, then we can extend this to a $(p,p')$ estimate for the operator given by convolution with $K_0^{\epsilon}$ via currying out the time component and using fractional integration. \\
One of the easiest ways to get a $(p,p')$ estimate is to interpolate between a $(1,\infty)$ estimate and a $(2,2)$ estimate. The previous estimate (\ref{1,inf partial estimate}) immediately allows us to conclude
\begin{equation}\label{(1,infty) bound}
\Vert T_t^{(\epsilon)}(g)\Vert_{L^{\infty}(\BR^{n-1})}\lesssim (1+|t|)^{\frac{2-n}{2}+1}\Vert g\Vert_{L^1(\BR^{n-1})}.
\end{equation}
We now search for a $(2,2)$ estimate for $T_t^{(\epsilon)}$ in order to interpolate and obtain a $(p,p')$ estimate. We use the fact that since $T_t^{(\epsilon)}$ is of convolution type, specifically with $K_0^{\epsilon}(\cdot,t)$, then
\begin{equation}
\Vert T_t^{(\epsilon)}\Vert_{L^2(\BR^{n-1})\rightarrow L^2(\BR^{n-1})}=\Vert\mathcal{F}_x[K_j^{\epsilon}](\cdot,t)\Vert_{L^{\infty}(\BR^{n-1})}
\end{equation}
where $\mathcal{F}_x$ denotes the Fourier transform in the spatial variables only. By linearity, we will instead obtain a bound on
\begin{equation}
\mathcal{F}_x[\check{\sigma}_s(\cdot,t)](\xi)=\int_{\BR^{n-1}}e^{-2\pi ix\cdot \xi}\check{\sigma}_s^{(0)}(x,t)\,dx
\end{equation}
for fixed $t$. We proceed by a calculation, recalling the notation $h_s(u)=\sqrt{|u|^2-s}$ and $k_s(u)$ as the derivative factor from pulling back to Euclidean space.
\begin{align}
\check{\sigma}_s(x,t)&=\int_{g^{-1}(s)}\eta(u,\tau)e^{2\pi i(u\cdot x+\tau t)}d\sigma_s(u,\tau) \\
&=\int_{\BR^{n-1}}e^{2\pi i th_s(u)}k_s(u)\eta(u,h_s(u))e^{2\pi iu\cdot x}du \\
&=\mathcal{F}_x^{-1}[e^{2\pi i th_s(\cdot)}k_s(\cdot)\eta(\cdot,h_s(\cdot))](x)
\end{align}
Thus, by Fourier inversion, we have that
\begin{equation}
\mathcal{F}_x[\check{\sigma}_s(\cdot,t)](\xi)=e^{2\pi i th_s(\xi)}k_s(\xi)\eta(\xi,h_s(\xi))
\end{equation}
which is uniformly bounded in $\xi,t$. We can differentiate in $s$ to calculate that
\begin{equation}
\left|\partial_s^2\int_{\BR^{n-1}}e^{-2\pi i\xi\cdot x}\check{\sigma}_s^{(0)}dx\right|\lesssim (1+|t|)^2.
\end{equation}
As before, we use the mean value theorem with the previous to estimates to immediately derive the bound
\begin{equation}
\left|\int_{\BR^{n-1}}e^{-2\pi i\xi\cdot x}(\check{\sigma}_s^{(0)}-2\check{\sigma}_0^{(0)}+\check{\sigma}_{-s}^{(0)})dx \right|\lesssim \min\big((1+|t|)^2s^2,1\big).
\end{equation}
Since all of these estimates have been uniform in $\xi$, we can make an analogous argument and calculation as the one used in the $(1,\infty)$ estimate. Letting $c_t\approx (1+|t|)^{-1}$, we have
\begin{align}
|\mathcal{F}_x[K_j^{\epsilon}](\xi,t)|&\lesssim(1+|t|)^{2}\int\limits_{|s|<c_t}ds+\int\limits_{\delta>|s|>c_t}\frac{ds}{s^2} \\
&\lesssim (1+|t|)\label{1,inf partial estimate}.
\end{align}
Hence, we have the $(2,2)$ estimate
\begin{equation} \label{(2,2) bound}
\Vert T_t^{(\epsilon)}\Vert_{L^2(\BR^{n-1})\rightarrow L^2(\BR^{n-1})}\lesssim (1+|t|).
\end{equation}
Interpolating between (\ref{(1,infty) bound}) and (\ref{(2,2) bound}), we get that for $1\leq p\leq 2$
\begin{equation}\label{(p,p') bound}
\Vert T_t^{(\epsilon)}\Vert_{L^p\rightarrow L^{p'}}\lesssim (1+|t|)^{\frac{2-n}{2}\left(\frac{2}{p}-1\right)+1}.
\end{equation}
Now, we will use this to show $(p,p')$ boundedness of the operator given by convolution with $K_0^{\epsilon}$ on $\BR^n$. Denote $k_{np}=\frac{2-n}{2}\left(\frac{2}{p}-1\right)+1$. We have
\begin{align}
\Vert K_0^{\epsilon}* P_0f\Vert_{L^{p'}(\BR^n)}&=\left\Vert\left\Vert\int_{v\in\BR} T_{t-v}(P_0f)(\cdot)dv\right\Vert_{L^{p'}(\BR^{n-1})}\right\Vert_{L^{p'}(\BR)} \\
&\lesssim \left\Vert\int_{v\in\BR}\left\Vert T_{t-v}(P_0f)(\cdot)\right\Vert_{L^{p'}(\BR^{n-1})}dv\right\Vert_{L^{p'}(\BR)} \\
&\lesssim \left\Vert\int_{v\in\BR} (1+|t-v|)^{k_{np}}\Vert P_0f(\cdot,v)\Vert_{L^p(\BR^{n-1})}dv\right\Vert_{L^{p'}(\BR)} \\
&\lesssim \left\Vert\int_{v\in\BR} \frac{1}{(1+|t-v|)^{2-\frac{2}{p}}}\Vert P_0f(\cdot,v)\Vert_{L^p(\BR^{n-1})}dv\right\Vert_{L^{p'}(\BR)} \\
&\lesssim \left\Vert (-\Delta)^{-\frac{1}{2}\left(\frac{2}{p}-1\right)}\left(\Vert P_0f(\cdot,t)\Vert_{L^p(\BR^{n-1})}\right)\right\Vert_{L^{p'}(\BR)}
\end{align}
where the quantity on the last line inside the $p'$ norm is the Riesz potential on the real line. Here, we used the fact that since $p\leq \frac{2n}{n+4}$, it follows that $1+\frac{2-n}{2}\left(\frac{2}{p}-1\right)\leq \frac{2}{p}-2$. Now, we can apply Hardy-Littlewood-Sobolev fractional integration because of the relation $\frac{1}{p'}=\frac{1}{p}-\left(\frac{2}{p}-1\right)$. We conclude that
\begin{equation}
\Vert K_0^{\epsilon}*P_0f\Vert_{L^{p'}(\BR^n)}\lesssim \Vert P_0f\Vert_{L^p(\BR^n)}.
\end{equation}
Using H\"older's inequality, we obtain the bound
\begin{equation} \label{Holder estimate}
\langle K_0^{\epsilon}*P_0f,P_0f\rangle \lesssim \Vert K_0^{\epsilon}*P_0f\Vert_{L^{p'}(\BR^n)}\Vert P_0f\Vert_{L^p(\BR^n)}\lesssim \Vert P_0f\Vert_{L^p(\BR^n)}^2.
\end{equation}
Thus, combining  (\ref{Convolution form}) and (\ref{Holder estimate}), we find that
\begin{equation} \label{Weak bound near cone}
\int_{L}\frac{|\widehat{P_0f}(\xi,\tau)|^2}{(|\xi|^2-\tau^2)^2+\epsilon^2}d\xi\,d\tau \lesssim \Vert P_0f\Vert_{L^{p}(\BR^n)}^2
\end{equation}
with an implied constant that is uniform in $\epsilon$; consequently, the monotone convergence theorem implies that
\begin{equation} \label{Bound near cone}
\int_{L}\frac{|\widehat{P_0f}(\xi,\tau)|^2}{(|\xi|^2-\tau^2)^2}d\xi\,d\tau \lesssim \Vert P_0f\Vert_{L^{p}(\BR^n)}^2.
\end{equation}
Using (\ref{Bound away from cone}) with (\ref{Bound near cone}), we conclude
\begin{equation}\label{Complete estimate away from 0}
\Vert P_0u\Vert_{L^2(\BR^n)}^2=\int_{L}\frac{|\widehat{P_0f}(\xi,\tau)|^2}{(|\xi|^2-\tau^2)^2}d\xi\,d\tau+\int_{A\setminus L}\frac{|\widehat{P_0f}(\xi,\tau)|^2}{(|\xi|^2-\tau^2)^2}d\xi\,d\tau \lesssim \Vert P_0f\Vert_{L^{p}(\BR^n)}^2
\end{equation}
which is the desired estimate.
\end{proof}
Having established the bound on a single annulus, we seek to extend this into a bound on all of Euclidean space; this will allow us to solve the PDE. The idea is to use the homogeneity of the wave equation to deduce bounds on the dyadic dilates of the annulus $A$ and then to sum over all the pieces and appeal to the Littlewood-Paley Theorem. In order for the sum to converge, we will need to take an appropriate Riesz potential of the solution $u$. As in the proof of the previous proposition, we will take $f\in L^1\cap L^p$ for simplicity.
\begin{thm} 
Suppose $f\in L^p(\BR^n)$ where $n\geq 5$ and $1< p\leq \frac{2n}{n+4}$ and that its Fourier transform $\hat{f}$ vanishes on the cone $\{(\xi,\tau)\in\BR^{n-1}\times\BR:|\xi|^2-\tau^2=0\}$. Then the inhomogeneous wave equation
\begin{equation}
u_{tt}-\Delta_xu=f,
\end{equation}
where $\Delta_x$ denotes the Laplacian in the $n-1$ spatial variables, has a solution $u$ which is a tempered distribution whose Fourier transform coincides with a function, which is the unique solution with the property
\begin{equation}
    (-\Delta)^{\frac{1}{4}\left(n+4-\frac{2n}{p}\right)}u\in L^2(\BR^n),
\end{equation}
where $(-\Delta)^z$ denotes the Riesz potential. The Fourier transform of the Riesz potential of the solution is given by
\begin{equation}
\left[(-\Delta)^{\frac{1}{4}\left(n+4-\frac{2n}{p}\right)}u\right]^{\land}(\xi,\tau)=(|\xi|^2+\tau^2)^{\frac{1}{4}\left(n+4-\frac{2n}{p}\right)}\frac{\hat{f}(\xi,\tau)}{|\xi|^2-\tau^2}.
\end{equation}
In particular, if $p=\frac{2n}{n+4}$ then $u\in L^2$.
\end{thm}
\begin{proof}
The form that the Riesz potential of $u$ must take is immediate by taking the Fourier transform of the PDE. Now, we will use (\ref{Complete estimate on annulus}) to derive an estimate that is useful for all values of $j$. We can exploit the homogeneity of the wave operator by changing variables. First, let $f_j(x,t)=f(2^jx,2^jt)$; this is clearly also an $L^p$ function. It follows that
\begin{equation}
    P_{-j}f(x,t)=(P_0f_j)(2^{-j}x,2^{-j}t).
\end{equation}
Changing variables, we obtain the identity
\begin{equation}
\Vert P_{-j}f\Vert_{L^p(\BR^n)}^2=2^{2jn/p}\Vert P_0f_j\Vert_{L^p(\BR^n)}^2.
\end{equation}
Now, let $u_j$ be the proposed solution for the inhomogeneous data $f_j$, i.e.
\begin{equation}
\hat{u}_j(\xi,\tau)=\frac{\hat{f_j}(\xi,\tau)}{|\xi|^2-\tau^2}.
\end{equation}
Then by the properties of the Fourier transform for dilations, we have
\begin{equation}
\widehat{P_0u_j}(\xi,\tau)=2^{-(n+2)j}\frac{\hat{f}(2^{-j}\xi,2^{-j}\tau)}{|2^{-j}\xi|^2-(2^{-j}\tau)^2}\chi(\xi,\tau).
\end{equation}
From this formula, changing variables gives the identity
\begin{equation}
\Vert P_{-j}u\Vert_{L^2(\BR^n)}^2=2^{(n+4)j}\Vert P_0u_j\Vert_{L^2(\BR^n)}^2.
\end{equation}
Combining all of this, we get the estimate
\begin{equation}\label{Estimate near origin}
\Vert P_{j}u\Vert_{L^2(\BR^n)}^2\lesssim 2^{-j(n+4)}\Vert P_0f_{-j}\Vert_{L^p(\BR^n)}^2=2^{-j\left(n+4-\frac{2n}{p}\right)}\Vert P_{j}f\Vert_{L^p(\BR^n)}^2.
\end{equation}
Now, using the Littlewood-Paley Theorem, as described in Theorem 6.1.2 of \cite{MR3243734} for instance, we conclude
\begin{align}
    \Vert (-\Delta)^{\frac{1}{4}\left(n+4-\frac{2n}{p}\right)}u\Vert_{L^2}&\lesssim\left(\sum_{j\in\BZ}2^{j\left(n+4-\frac{2n}{p}\right)}\Vert P_j u\Vert_{L^2}^2\right)^{1/2} \\
    &\lesssim\left(\sum_{j\in\BZ}\Vert P_j f\Vert_{L^p}^2\right)^{1/2} \\
    &\lesssim\left\Vert\left(\sum_{j\in\BZ}|P_jf|^2\right)^{1/2}\right\Vert_{L^p}\\
    &\lesssim \Vert f\Vert_{L^p}.
\end{align}
This completes the proof.
\end{proof}
\section{Sharpness of Proposition 2}
In this, section, we will construct an example showing that the condition on the exponent $p$ where $p\leq \frac{2n}{n+4}$ is sharp on an annulus for $n\geq 5$. More specifically, we will show that the estimate
\begin{equation}\label{sharpness equation}
\Vert P_0u\Vert_{L^2(\BR^n)}^2\lesssim \Vert P_0f\Vert_{L^p(\BR^n)}^2
\end{equation}
fails for $p>\frac{2n}{n+4}$, where $P_0$ denotes the Littlewood-Paley projection from the previous section. To do this, it suffices to construct a family of functions $f_{\epsilon}$ with Fourier support on an annulus and estimate the dependence on $\epsilon$ as it tends to 0. \\
Define $\varphi_{a,b,c}:\BR\rightarrow \BR$ to be a bump function which is identically one on $(a,b)$ and vanishes outside $(a-c,b+c)$. Define the function $\hat{f}_1:\BR^n\rightarrow\BR$ by
\begin{equation}
\hat{f}_1(\xi_1,...,\xi_{n-1},\tau)=\varphi_{0.55,0.95,0.05}(\xi_1-\tau)\varphi_{1.1,1.9,0.1}(\xi_1+\tau)\prod_{j=2}^{n-1} \varphi_{0.1,0.9.0,1}(\xi_j)
\end{equation}
Under this definition, $\hat{f}_{1}$ is a smooth bump function supported on a rectangular prism which has one vertex at the point $(1,0,...,0,1)$, with sides of length 1 in the $\xi_1+\tau$ direction, the $\xi_1-\tau$ direction, and the $\xi_j$ direction for $j=2,...,n-1$. By construction, this vanishes on the cone where $|\xi|=|\tau|$. Now, we will scale this about the vertex at $(1,0,...,0,1)$ by translating to the origin, dilating, and translating back to the vertex $(1,0,...,0,1)$. Thus, if we let $\mathcal{T}$ represent the translation operation and $L_{\epsilon}$ the dilation operation, precisely the unique linear operator which dilates the argument of $f$ about the origin by a factor of $\epsilon$ in the directions $\xi_2,...,\xi_{n-1}$ and dilates by a factor of $\epsilon^2$ in the direction $\xi_1-\tau$ (note there is no dilation in the $\xi_1+\tau$ direction), then we can write
\begin{equation}
    \hat{f}_{\epsilon}=(\mathcal{T}^{-1}\circ L_{\epsilon}\circ\mathcal{T})(\hat{f}_{1/2})
\end{equation}
We have that $\Vert f_{1}\Vert_{L^p(\BR^n)}^2=C\approx 1$ since the inverse Fourier transform of a Schwartz function is Schwartz, and these have finite $L^p$ norms for all $p$.
Taking the Fourier inverse transform of $\hat{f}_{\epsilon}$, the translations give complex exponential factors with norm 1, and the dilation gives a factor of $\epsilon^n$ while dilating $f$ in the other direction. Specifically, this dilation operator is a linear transform $L_{\epsilon}$ with determinant $\epsilon^{-n}$ whose inverse is $L_{1/\epsilon}$, so we have
\begin{equation}
    |f_{\epsilon}(x,t)|=\epsilon^n|f_{1}(L_{1/\epsilon}(x,t))|.
\end{equation}
With this it is easy to compute by changing variables that
\begin{equation}
\Vert f_{\epsilon}\Vert_{L^p(\BR^n)}^2=\epsilon^{2n}\left(\epsilon^{-n}\int_{\BR^n}|f_{1}(x,t)|^p],dxdt\right)^{2/p}=\epsilon^{\frac{2n}{p'}}\Vert f_{1}\Vert_{L^p(\BR^n)}^2.
\end{equation}
Now, we will estimate $\Vert \hat{u}_{\epsilon}\Vert_{L^2(\BR^n)}^2$. We have that
\begin{equation}
    \Vert \hat{u}_{\epsilon}\Vert_{L^2(\BR^n)}^2=\int \frac{|\hat{f}_1(\xi,\tau)|^2}{|\xi_2^2+...+\xi_{n-1}^2+(\xi_1-\tau)(\xi_1+\tau)|^2}d\xi\,d\tau
\end{equation}
On the support of $\hat{f}_1$, we can bound $\xi_1+\tau$ between $1$ and $2$. Making the change of variables $u=\xi_1-\tau,v=\xi_1+\tau$ we get
\begin{equation}
    \Vert \hat{u}_{\epsilon}\Vert_{L^2(\BR^n)}^2\gtrsim \epsilon^{n-2}\int_{.55\epsilon^2}^{.95\epsilon^2}\frac{du}{(u+(n-2)\epsilon^2)^2}\approx \epsilon^{n-4}
\end{equation}
where we used the upper bound of $\epsilon^2$ for all the squared terms in the summation. Thus, for the estimate to hold, we would need
\begin{equation}
\epsilon^{n-4}\lesssim \epsilon^{\frac{2n}{p'}}
\end{equation}
Letting $\epsilon$ go to 0, this means that $n-4\geq\frac{2n}{p'}$ which is equivalent to the condition that $p\leq \frac{2n}{n+4}$. Hence, this condition is necessary for the inequality (\ref{sharpness equation}) to hold. \\
In fact, we can deduce another sharpness result using the same example. Letting $\epsilon$ go to infinity in the previous example, we see that estimates of the form
\begin{equation}\label{Final bound}
\Vert u \Vert_{L^2(\BR^n)}^2\lesssim \Vert f \Vert_{L^p(\BR^n)}^2
\end{equation}
are not possible unless $p\geq \frac{2n}{n+4}$. Thus, the only value of $p$ for which we get an estimate of the form (\ref{Final bound}) is when $p=\frac{2n}{n+4}$.

\bibliographystyle{unsrt}
\bibliography{main}

\end{document}